\documentclass[10pt]{amsart}
\usepackage[utf8]{inputenc}
\usepackage{amsmath}
\usepackage{amssymb}
\usepackage[all]{xy}

\makeatletter
\@namedef{subjclassname@2010}{
 \textup{2010} Mathematics Subject Classification}
\makeatother

\newtheorem{Thm}{Theorem}[section]
\newtheorem{Prp}[Thm]{Proposition}
\newtheorem{Cor}[Thm]{Corollary}
\newtheorem{Lma}[Thm]{Lemma} 
\theoremstyle{definition}
\newtheorem{Def}[Thm]{Definition}
\newtheorem{Rk}[Thm]{Remark}

\numberwithin{equation}{section}

\newcommand{\LBi}{\mathfrak{LBi}}
\newcommand{\fD}{\mathfrak{D}}
\newcommand{\ad}{\mathrm{ad}}
\newcommand{\bGL}{\mathbf{GL}}
\newcommand{\Spec}{\mathrm{Spec}}

\newcommand{\End}{\mathrm{End}}
\newcommand{\Aut}{\mathrm{Aut}}
\newcommand{\Hom}{\mathrm{Hom}}
\newcommand{\bAut}{\mathbf{Aut}}

\newcommand{\KK}{\mathbb{K}}
\newcommand{\LL}{\mathbb{L}}

\newcommand{\Gal}{\mathrm{Gal}}
\newcommand{\sll}{\mathfrak{sl}}
\newcommand{\su}{\mathfrak{su}}
\newcommand{\g}{\mathfrak{g}}
\newcommand{\bo}{\mathfrak{b}}
\newcommand{\h}{\mathfrak{h}}
\newcommand{\Ad}{\mathrm{Ad}}

\newcommand{\bA}{\mathbf{A}}
\newcommand{\bE}{\mathbf{E}}
\newcommand{\bH}{\mathbf{H}}
\newcommand{\bC}{\mathbf{C}}
\newcommand{\bG}{\mathbf{G}}
\newcommand{\bbQ}{\mathbb{Q}}

\newcommand{\bbZ}{\mathbb{Z}}
\newcommand{\Id}{\mathrm{Id}}
\newcommand{\CYB}{\mathrm{CYB}}
\newcommand{\drj}{\mathrm{DJ}}
\newcommand{\bdr}{\mathrm{BD}}
\newcommand{\Lie}{\mathrm{Lie}}
\newcommand{{\Kb}}{\overline{K}}
\newcommand{{\dr}}{\partial r}

\begin{document}
\title{On the Classification of Lie Bialgebras by Cohomological Means}
\author[S.\ Alsaody and A.\ Pianzola]{Seidon Alsaody$^1$ and Arturo Pianzola$^1$ $^2$}
\address{$^1$ Department of Mathematical and Statistical Sciences, University of Alberta, Edmonton, AB T6G 2G1, Canada.}
\address{$^2$ Centro de Altos Estudios en Ciencias Exactas, Avenida de Mayo 866, (1084), Buenos Aires, Argentina.}
\email{seidon.alsaody@gmail.com, a.pianzola@gmail.com}
\thanks{S.\ Alsaody wishes to thank the Knut and Alice Wallenberg Foundation for the grant KAW 2015.0367, by means of which he was supported as a postdoctoral researcher at 
Institut Camille Jordan, Lyon. A.\ Pianzola wishes to thank NSERC and Conicet for their continuing support.}

\begin{abstract} We approach the classification of Lie bialgebra structures on
simple Lie algebras from the viewpoint of descent and non-abelian cohomology. We achieve a description of the problem
in terms faithfully flat cohomology over an arbitrary ring over $\mathbb{Q}$, and solve it for Drinfeld--Jimbo Lie bialgebras over fields of characteristic zero. We consider the classification up
to isomorphism, as opposed to equivalence, and treat split and non-split Lie algebras alike. We moreover give a new interpretation of scalar multiples of Lie bialgebras hitherto studied using twisted Belavin--Drinfeld cohomology. 
\end{abstract}

\subjclass[2010]{17B62, 17B37, 20G10}
\keywords{Lie bialgebra, quantum group, faithfully flat descent, Galois cohomology}
\maketitle

\section{Introduction}
The ``linearization problem" for quantum groups, outlined in spirit by Drinfeld \cite{D}, and solved in the celebrated work of Etingof and  Kazhdan 
\cite{EK1} and \cite{EK2}, naturally leads to the study of Lie bialgebras over the power series ring $R = \mathbb{C}[[t]].$ If $\g$ is a finite dimensional (necessarily) 
split simple complex Lie algebra one can try to understand all possible Lie bialgebra structures that can be put on the $R$--Lie algebra $\g \otimes_\mathbb{C} R$. This is exactly the program started by Kadets, Karolinsky, Pop and Stolin and
pursued in \cite{KKPS} and other papers, where this is done by considering  the (algebraic) Laurent series field $\KK = \mathbb{C}((t))$, and introducing (twisted and untwisted) Belavin--Drinfeld cohomologies. These cohomologies parametrized the possible Lie bialgebra structures on $\g \otimes_\mathbb{C} \KK,$ and they were computed on a case-by-case fashion for the classical types.

We recall that the Belavin--Drinfeld theorem from \cite{BD} gives a complete list (up to equivalence) of all possible Lie bialgebra structures on 
$\g \otimes_\mathbb{C} \overline{\KK}.$ It is thus natural to approach the problem at hand by means of usual Galois cohomology. This is done in \cite{PS}, 
where the Belavin--Drinfeld cohomologies are shown to be usual Galois cohomology with values in the algebraic group $\bC(\bG, r_\bdr)$  
(the centralizer in the adjoint group $\bG$ of $\g$ of the given Belavin--Drinfeld matrix $r_\bdr$).
The main results of \cite{PS} state that:
\medskip

(a) untwisted Belavin--Drinfeld cohomologies are the usual Galois cohomologies $H_\Gal^1(\KK, \bC(\bG, r_\bdr).$\

(b) for the Drinfeld--Jimbo $r$-matrix $r_{\rm DJ}$ the twisted Belavin--Drinfeld cohomologies can be expressed in terms of the Galois cohomology set $H_\Gal^1(\KK, \widetilde{\bC}
(\bG, r_\drj)$ for a twisted form $\widetilde{\bC}(\bG, r_\drj)$ of the $\KK$-algebraic group $\bC(\bG, r_{\rm DJ})$. This form is split by the quadratic extension $\KK(\sqrt{t})$ of $\KK$.

(c) $H_\Gal^1(\KK, \bC(\bG, r_\drj)$ is trivial by Hilbert 90, and  $H^1(\KK, \widetilde{\bC}(\bG, r_\drj)$ is trivial by a theorem of Steinberg that is also 
used to establish the correspondence in (b). As a consequence, there are unique (up to Belavin--Drinfeld equivalence with gauge group $\bG$) corresponding 
Lie bialgebra structures on $\g$ with prescribed doubles (namely $\g \times \g$ in the untwisted case, and $\g \otimes_\KK \LL$ in the twisted case).
 \medskip

The main objective of the present paper is to develop the theory of faithfully flat descent for Lie bialgebras over rings, with emphasis on what this theory is best suited for: 
the classification of twisted forms of a given Lie bialgebra {\it up to isomorphism} and \emph{without the restriction that the underlying Lie algebras be split}. This is the main
difference between our work and the recent paper \cite{KPS}, where the authors use Galois descent to obtain far-reaching results about Lie bialgebra structures on split Lie algebras up to equivalence.
The Belavin--Drinfeld classification is up to equivalence, in the sense where two coboundary Lie bialgebra structures $\partial r$ and $\partial r'$ on a split Lie algebra $\g$ are 
considered equivalent if
\[r'=\alpha(\Ad_X\otimes\Ad_X)(r)\]
for some invertible scalar $\alpha$ and some $X$ in a suitably chosen group with corresponding simple Lie algebra $\g$. This relation is not comparable to isomorphism. 
On the one hand, scalar multiples of
$r$-matrices in general lead to non-isomorphic Lie bialgebras. On the other hand, non-equivalent Lie bialgebra structures may still be isomorphic if the Lie algebra admits 
outer automorphisms. Nevertheless, the flexibility of the point of view that we take allows us to recover and in some cases explain all the results up to equivalence known heretofore.
As we will see, it is well suited for understanding Lie bialgebras whose underlying Lie algebra is not necessarily split over arbitrary fields of characteristic zero, a topic that has been little investigated in the literature.

In the appendix of \cite{KPS}, the authors classify Lie bialgebra structures on $\sll(2,R)$ and $\sll(3,R)$ for a discrete valuation ring $R$ using orders and lattices. It would
therefore be interesting to find a cohomological interpretation and generalization of these results. Another instance where Lie bialgebras over rings are discussed
is \cite{BFS}, where solutions to the quantum Yang--Baxter equation that arise from Frobenius algebras over rings are treated.

After the necessary definitions in this section, the paper proceeds with the statement of the formalism of faithfully flat descent for Lie bialgebras in Section 2. In Section 3 we give a 
description of the automorphism groups of Belavin--Drinfeld Lie bialgebras defined over the base ring. Our major result is in Section 4, where we solve the classification problem for standard
Lie bialgebras over arbitrary fields of characteristic zero. In Section 5 we turn our attention to Lie bialgebras that are locally scalar multiples of Belavin--Drinfeld structures. This includes
and provides context to previous results on twisted Belavin--Drinfeld bialgebras. In the final Section 6 we review some known classification results in the light of the results
of the previous sections.

\subsection*{Acknowledgement} We are grateful to Alexander Stolin for fruitful discussions.

\subsection{Lie Bialgebras over Rings}
The importance of considering Lie bialgebras over rings that are not fields was explained in the introduction. Throughout, we fix a unital, commutative ring $R$. All unadorned tensor products are understood to be over $R$. By an
\emph{$R$-ring} we understand a unital, commutative $R$-algebra. We will further always assume that $\Spec\ R$, as a scheme,
has characteristic 0; this amounts to saying that $R$ is a $\bbQ$-ring. For any $R$-module $M$ we will always write $\kappa$ for the linear map $M\otimes M\to M\otimes M$ defined by the transposition $x\otimes y\mapsto y\otimes x$ of tensor factors. Let $M$ be an $R$-module. A
\emph{Lie cobracket} on $M$ is an $R$-linear map $\delta:M\to M\otimes M$ that is \emph{anti-symmetric} in the sense that
\[\kappa\circ\delta=-\delta,\]
and satisfies the \emph{co-Jacobi identity}
\[(\delta\otimes\Id_M)\circ\delta=(\Id_M\otimes\delta)\circ\delta+(\Id_M\otimes\kappa)\circ(\delta\otimes\Id_M)\circ\delta.\]
The pair $(M,\delta)$ is called a \emph{Lie coalgebra}. From the definition it follows that the composition
\[\xymatrix{M^*\otimes M^*\ar[r]^-{\mathrm{can}}&(M\otimes M)^*\ar[r]^-{\delta^*}&M^*}\]
is a Lie bracket on $M^*$.

\begin{Rk} If $M$ is a finitely generated projective module, then the canonical map $\mathrm{can}$ is an isomorphism. In that case,
identifying $M^*\otimes M^*$ with $(M\otimes M)^*$, it can be verified that $\delta$ is a Lie cobracket if and only if $\delta^*$ is a Lie bracket.
\end{Rk}

If $\g=(\g,[,])$ is a Lie algebra, and $\delta$ is a Lie cobracket on $\g$ satisfying the
\emph{cocycle condition}
\[\delta([a,b])=(\ad_a\otimes 1+1\otimes\ad_a)\delta(b)-(\ad_b\otimes
1+1\otimes\ad_b)\delta(a)\]
for any $a,b\in \g$, then $(\g,\delta)$ is called a \emph{Lie bialgebra}. If $(\g,\delta)$ and $(\g',\delta')$ are Lie bialgebras, then a map
$\phi: \g\to\g'$ is a morphism of Lie bialgebras $(\g,\delta)\to(\g',\delta')$ if it
is a Lie algebra morphism that in addition satisfies
\begin{equation}\label{morphism} (\phi\otimes\phi)\circ\delta=\delta'\circ\phi. \end{equation}
If $\phi$ is invertible, then $\phi^{-1}$ is a morphism of Lie bialgebras $(\g',\delta')\to(\g,\delta)$, and $\phi$ is called an isomorphism of Lie bialgebras. 
Thus $R$-Lie bialgebras form a category, which we denote by $\widehat\LBi_R$. We denote by $\LBi_R$ the full subcategory of $\widehat\LBi_R$ whose objects
are those Lie bialgebras whose underlying module is finitely generated and projective.

\section{Descent for Lie Bialgebras}
We will first establish the desired correspondence between twisted forms of
bialgebras and certain cohomology classes. Let $(\g,\delta)$ be a Lie
bialgebra over $R$, and let $S$ be a $R$-ring. On the $S$-algebra $\g_S=\g\otimes S$ one has a unique Lie bialgebra
structure $\delta_S$ that satisfies
\[\delta_S(x\otimes 1)=\sum (y_i\otimes 1)\otimes_S (z_i\otimes 1)\]
for all $x\in \g$, where $\sum y_i\otimes z_i=\delta(x)\in\g\otimes\g$. An
$R$-Lie bialgebra $(\g',\delta')$ is said to be an \emph{$S/R$-twisted form of
$(\g,\delta)$} if
$(\g'_S,\delta'_S)\simeq (\g_S,\delta_S)$ as $S$-bialgebras. We will mainly be interested in the case where $S$ is faithfully flat over $R$.
This includes the special case where $R$ is a field of characteristic zero and $S$ is any field extension.

Let $(\g,\delta)$ now be an $S$-Lie bialgebra, and let $\kappa:S\otimes S\to S\otimes S$ be the ($R$-linear) flip $\alpha\otimes\beta\mapsto\beta\otimes\alpha$. There are two ways to endow $\g'\otimes S$ with an $S\otimes S$-module structure; the $S\otimes S$-action
being component-wise in the first, and twisted by $\kappa$ in the second. We denote the two modules (algebras, bialgebras) by $\g'\otimes_{12} S$ and $\g'\otimes_{21} S$, respectively. Both modules (algebras, bialgebras) are seen as having the same underlying $R$-module
($R$-algebra, $R$-bialgebra) structure which we continue to denote by $\g'\otimes S$.

\begin{Def}A \emph{descent datum} on $\g$ is an isomorphism $\theta:\g\otimes_{12} S\to \g\otimes_{21} S$ of $S\otimes S$- Lie bialgebras, satisfying the equality
\[(\Id\otimes\kappa)(\theta\otimes\Id)(\Id\otimes\kappa)=(\theta\otimes\Id)(\Id\otimes\kappa)(\theta\otimes\Id)\]
of maps on $\g\otimes S\otimes S$. The triple $(\g,\delta,\theta)$ is called a \emph{Lie bialgebra with a descent datum}. A \emph{morphism of Lie
bialgebras with descent data} $(\g,\delta,\theta)\to(\g',\delta',\theta')$ is a morphism of Lie bialgebras $f:(\g,\delta)\to(\g',\delta')$ such that the diagram
\[\xymatrix{
\g\otimes_{12} S\ar[r]\ar[d]_\theta&\g'\otimes_{12} S\ar[d]^{\theta'}\\ \g\otimes_{21} S\ar[r]&\g'\otimes_{21} S}\]
commutes, where the horizontal arrows are given by $f\otimes \Id$.
\end{Def}

\begin{Rk}
In the literature, it is customary to set $\theta^0=(\Id\otimes\kappa)(\theta\otimes\Id)(\Id\otimes\kappa)$, $\theta^1=(\theta\otimes\Id)(\Id\otimes\kappa)$
and $\theta^2=(\theta\otimes\Id)$, and write the equality in the definition as $\theta^1=\theta^0\theta^2$.
\end{Rk}

Here and in what follows, if $N$ is an $R$-module and $T$ an $R$-ring, we shall often abbreviate $N\otimes T$ by $N_T$. For an $R$-ring $S$, we write $\widehat\LBi_R^S$ for the category of $S$-Lie bialgebras with descent data, and $\LBi_R^S$ for the full subcategory
formed by the objects whose underlying modules are finitely generated projective. If $(\g,\delta)\in\widehat\LBi_R$,
then the \emph{standard descent datum} on $\g\otimes S$ is the map
\[\Id_\g\otimes\kappa:(\g\otimes S)\otimes_{12} S\to(\g\otimes S)\otimes_{21} S.\]
It is straight-forward to verify that this is indeed a descent datum, and that we thus get a functor $\fD=\fD_R^S:\widehat\LBi_R\to\widehat\LBi_R^S$,
defined on objects by $\g\mapsto (\g\otimes S,\Id_\g\otimes\kappa)$, and on morphisms by $f\mapsto f\otimes\Id_S$.

A straightforward but delicate reasoning yields the expected faithfully flat descent theory for Lie bialgebras as follows. 
(See the preprint \cite{AP} of this paper for the detailed proof.)

\begin{Prp}\label{Pdescent} If the $R$-ring $S$ is faithfully flat, then $\fD$ is an equivalence of categories $\widehat\LBi_R\to\widehat\LBi_R^S$, and induces an equivalence $\LBi_R\to\LBi_R^S$.
\end{Prp}

For any $R$-Lie bialgebra $(\g,\delta)$ and any faithfully flat $R$-ring $S$, we wish to classify all \emph{$S/R$-twisted forms of $\g$}, i.e.\ all $R$-Lie bialgebras $(\g',\delta')$ such that $(\g'_S,\delta'_S)\simeq(\g_S,\delta_S)$. Let $\bA=\bAut((\g,\delta))$ be the automorphism group functor of $(\g,\delta)$.
As one does for modules and algebras, we consider, for each faithfully flat $R$-ring $S$ the cohomology set $H^1(S/R,\bA):=H^1_{\mathrm{fppf}}(S/R,\bA)$, consisting of cohomology classes of cocycles, where a cocycle is an element $\phi\in\bA(S\otimes S)$ satisfying the cocycle condition, and where two cocycles $\phi$ and $\phi'$ are defined to be cohomologous if
\[\phi'=(\Id_\g\otimes\kappa)(\rho\otimes\Id_S)(\Id_\g\otimes\kappa)\phi(\rho^{-1}\otimes\Id_S)\]
for some $\rho\in\bA(S)$. The following is then a consequence of the above, and the proof is analogous to that for descent of modules.

\begin{Cor}\label{Cbij} Let $(\g,\delta)$ be a Lie bialgebra over $R$ with automorphism
group scheme $\bA$. Let $S$ be a faithfully flat $R$-ring. Then there is a $1-1$-correspondence between
$H^1(S/R,\bA)$ and $R$-isomorphism classes of $S/R$-twisted forms of $(\g,\delta)$.
\end{Cor}

\section{Belavin--Drinfeld Lie Bialgebras and Their Automorphisms}
\subsection{Coboundary Lie Bialgebras and $r$-matrices}
A Lie bialgebra $(\g,\delta)$ is said to be a \emph{coboundary Lie bialgebra} if $\delta=\dr$ for some $r\in\g\otimes\g$, i.e.\ if
\begin{equation}\label{coboundary} \delta(a)=(\ad_a\otimes 1+1\otimes\ad_a)(r) \end{equation}
for all $a\in\g$. Using classical notation, alluding to the universal enveloping algebra, this is written as
\[\delta(a)=[a\otimes 1+1\otimes a,r]\]
In general, not every $r\in\g\otimes\g$ gives rise to a Lie bialgebra structure via the above formula. We will come back to this point below. 

Let $\bG$ be a split simple adjoint group over $R$ with $\g=\Lie(\bG)$. By Chevalley uniqueness \cite[XXIII.5]{SGA3} and the fact that $R$ is a $\bbQ$-ring, up to isomorphism we may and will assume that $\bG$
is defined over $\bbQ$; thus $\g=\g_0\otimes_\bbQ R$ for a split simple $\bbQ$-Lie algebra $\g_0$. Let $\bE$ be a pinning of $\bG$. The pinning provides a 
split maximal torus $\bH$, a splitting Cartan subalgebra $\h=\Lie(\bH)$ of $\g$, a base $\Gamma$ of the corresponding
root system $\Delta$, a set of positive roots $\Delta^+\subset\Delta$ and, for each $\alpha\in\Delta$, a Chevalley generator $X_\alpha\neq 0$ of $\g_\alpha$. This moreover provides a Casimir
element $\Omega\in\g\otimes\g$, and we write $\Omega_\h$ for its Cartan part. (More precisely, writing $\h=\h_0\otimes_\bbQ R$, where $\h_0$ is the Cartan 
subalgebra of $\g_0$ corresponding to the above data, and taking an
orthonormal basis $(h_i)$ of $\h_0$,
$\Omega_\h$ is the image of $\sum_i h_i\otimes h_i$ under the base change $\bbQ\to R$.) 

\begin{Rk}\label{Omega} It is known that in the above setting
\begin{equation}\label{Q}
\bbQ\Omega=\{s\in\g_0\otimes\g_0|\forall a\in\g_0: (\ad_{a}\otimes 1+1\otimes\ad_{a})(s)=0\}, 
\end{equation}
and that any automorphism of $\g_0$ fixes $\Omega$. From this we will deduce that
\begin{equation}\label{R}
R\Omega=\{s\in\g\otimes\g|\forall a\in\g: (\ad_{a}\otimes 1+1\otimes\ad_{a})(s)=0\},
\end{equation}
and that any automorphism of $\g$ fixes $\Omega$. Indeed, consider the $\bbQ$-linear map
\[F_0: \g_0\otimes\g_0\to \Hom_\bbQ(\g_0,\g_0\otimes\g_0), s\mapsto \partial_s\]
where $\partial_s(a)=[1\otimes a+a\otimes 1,s]$. The kernel of $F_0$ is the right hand side of \eqref{Q}, which thus is $\bbQ\Omega$. Base change gives a map
\[F_0\otimes_\bbQ\Id_R:(\g_0\otimes_\bbQ\g_0)\otimes_\bbQ R\to \Hom_\bbQ(\g_0,\g_0\otimes_\bbQ\g_0)\otimes_\bbQ R,\]
where the right hand side is (canonically) isomorphic to $\Hom_R(\g,(\g_0\otimes_\bbQ\g_0)\otimes_\bbQ R)$ since $\g_0$ is finite-dimensional. Since $R$ is flat over $\bbQ$, the 
kernel of this map is $\bbQ\Omega\otimes_\bbQ R$. Identifying  $(\g_0\otimes_\bbQ\g_0)\otimes_\bbQ R$ with $\g\otimes_R\g$ we thus deduce that the kernel of the map
\[F:\g\otimes\g\to \Hom_R(\g,\g\otimes\g), s\mapsto \partial_s,\]
is $R\Omega$. Furthermore $\Aut(\g)=\Aut(\g_0)_R$, whence for any $R$-ring $S$ and any $\phi\in\Aut(\g)(S)$,
\[(\phi\otimes_S\phi)(\Omega\otimes_\bbQ 1_S)=\Omega\otimes_\bbQ 1_S\]
Note further that since $\g=\g_0\otimes_{\bbQ}R$ and $\Omega$ is defined over $\bbQ$ and non-zero in $\g_0\otimes\g_0$, 
it follows that if $\lambda\in R$ satisfies $\lambda\Omega=0$ in $\g\otimes\g$, then $\lambda=0$. 
\end{Rk}

By an \emph{$r$-matrix on $\g$} we understand an element $r\in\g\otimes\g$ satisfying $\CYB(r)=0$ and $r+\kappa(r)=\lambda\Omega$ for some $\lambda\in R$. Here
the classical Yang--Baxter operator $\CYB$ is defined by
\[\CYB(r)=[r_{12},r_{13}]+[r_{12},r_{23}]+[r_{13},r_{23}],\]
which, writing $r=\sum_i s_i\otimes t_i$, is shorthand for\footnote{The notation is motivated by the classical situation where one passes to the
universal enveloping algebra $U(\g)$ and sets e.g.\ $(s\otimes t)_{13}=s\otimes1\otimes t$. The bracket then denotes the commutator in $U(\g)^{\otimes3}$. This is however merely a convenient notation and there is no need to resort to
$U(\g)$.}
\[\sum_{i,j}\left([s_i,s_j]\otimes t_i\otimes t_j+s_i\otimes[t_i,s_j]\otimes t_j+s_i\otimes s_j\otimes[t_i,t_j]\right)\in\g\otimes\g\otimes\g.\]

It is straightforward to check that if $r$ is an $r$-matrix, then $\partial r$ is a Lie bialgebra structure on $\g$. Conversely, if $R=K$ is
an algebraically closed field and $\delta$ is a Lie bialgebra structure on $\g$, then $\delta=\partial r$ for some $r$-matrix $r$. 

\begin{Rk} When $R$ is a field, $r$-matrices $r$ satisfying $r+\kappa(r)=0$ are called skew-symmetric. These are excluded from the Belavin--Drinfeld classification. In the sequel we will only consider $r$-matrices satisfying $r+\kappa(r)=\lambda\Omega$ with 
$\lambda\in R^*$, i.e.\ those that remain non-skew symmetric under any base change $R\to K$ with $K$ a field. 
\end{Rk}

Below we list a few properties of $r$-matrices and their coboundary structures for later use.

\begin{Lma}\label{Lsurj} Let $r_1$ and $r_2$ be two $r$-matrices over $\g$, and let $\phi$ be a surjective endomorphism of the Lie algebra $\g$.
Then $\phi$ is a morphism of Lie bialgebras $(\g,\dr_1)\to(\g,\dr_2)$ if and only if $(\phi\otimes\phi)(r_1)-r_2\in R\Omega$.
\end{Lma}

\begin{proof} Set $s=(\phi\otimes\phi)(r_1)-r_2$. Combining \eqref{morphism} and \eqref{coboundary}, one sees that $\phi$ is a Lie bialgebra morphism if and only if $(\ad_{\phi(a)}\otimes 1+1\otimes\ad_{\phi(a)})(s)=0$ for all $a\in\g$. 
Remark \ref{Omega} then implies that $s\in R\Omega$, since $\phi$ is surjective.
\end{proof}

\begin{Lma}\label{Laut} Let $\phi\in\Aut(\g)$ and let $r$ be an $r$-matrix on $\g$. Then $\phi$ is an automorphism of $(\g,\dr)$ if and only if $(\phi\otimes\phi)(r)=r$.
\end{Lma}

\begin{proof} From the previous lemma we know that $(\phi\otimes\phi)(r)=r+\mu \Omega$ for some $\mu\in R$. 
Moreover, $r$ satisfies $r+\kappa(r)=\lambda\Omega$ for some $\lambda\in R$.
Thus
\[(\phi\otimes\phi)(r+\kappa(r))=(\phi\otimes\phi)(r)+(\phi\otimes\phi)\kappa(r)=r+\mu\Omega+\kappa(r)+\mu\Omega=(\lambda+2\mu)\Omega\]
while the left hand side equals $(\phi\otimes\phi)(\lambda\Omega)=\lambda\Omega$. Thus $2\mu\Omega=0$, whence $\mu=0$ by Remark \ref{Omega}.
\end{proof}

\begin{Lma}\label{Lelim} Assume that $R$ is an integral domain. Let, for $i=1,2$,
$r_1$ and $r_2$ be two $r$-matrices with $r_i+\kappa(r_i)=\lambda_i\Omega$ with $\lambda_i\in R^*$. If $r_2=r_1-\mu\Omega$ for some $\mu\in R$, then
either $\mu=0$ or $\mu=\lambda_1$.
\end{Lma}

The result and its proof over fields have been communicated to us by A.\ Stolin. The proof over integral domains is almost identical, namely by inserting $r_2=r_1-\mu\Omega$ into the equation $\CYB(r_2)=0$,
and simplifying the expression using the fact that $\CYB(r_1)=0$ and $\kappa(r_1)=\lambda_1\Omega-r_1$. Doing so, one sees that the equation is equivalent to $\mu(\mu-\lambda_1)[\Omega_{12},\Omega_{13}]=0$, where the subscripts are as in the definition of the 
classical Yang--Baxter operator. Now $[\Omega_{12},\Omega_{13}]$ is defined over $\bbQ$ and non-zero in $\g_0^{\otimes3}$, whence it is free over $R$. Thus $\mu(\mu-\lambda_1)=0$ and we conclude with the assumption on $R$.

\begin{Lma}\label{Iso} Assume that $R$ is an integral domain, let $r$ be an $r$-matrix on $\g$ with $r+\kappa(r)=\lambda\Omega$, $\lambda\in R^*$, and let $\alpha, \beta\in R^*$. If $(\g,\partial\alpha r)\simeq(\g,\partial\beta r)$, then $\beta=\pm\alpha$.  
\end{Lma}

\begin{proof} Assume that $\phi:(\g,\partial\alpha r)\simeq(\g,\partial\beta r)$ is an automorphism. By Lemma \ref{Lsurj}, $(\phi\otimes\phi)(\alpha r)=\beta r+\mu\Omega$ for some $\mu\in R$. Analogously to the proof of Lemma \ref{Laut}, we get
\[(\phi\otimes\phi)(\alpha r+\kappa(\alpha r))=(\beta\lambda+2\mu)\Omega,\]
while the left hand side equals $(\phi\otimes\phi)(\alpha\lambda\Omega)=\alpha\lambda\Omega$. Thus Remark \ref{Omega} gives $\beta\lambda=\alpha\lambda-2\mu$. On the other hand, $r'=(\phi\otimes\phi)(\alpha r)$ is an $r$-matrix with
$r'+\kappa(r')=\alpha\lambda\Omega$. Thus Lemma \ref{Lelim} implies that $\mu=0$ or $\mu=\alpha\lambda$. Inserting these cases into $\beta\lambda=\alpha\lambda-2\mu$ gives $\beta=\pm\alpha$, since $\lambda$ is invertible.
\end{proof}

Finally, for later use, we recall the split exact sequence
\begin{equation}\label{exact}
\xymatrix{
\mathbf{1}\ar[r]&\bG\ar[r]^-{\Ad}&\bAut(\g)\ar[r]^-{f}&\bA_{\Gamma}\ar[r]&\mathbf{1}}\end{equation}
of affine group schemes, where $\Ad$ denotes the adjoint representation and $\bA_{\Gamma}$ is the constant group scheme corresponding to the finite abstract group 
$\Aut(\Gamma)$.

\subsection{Belavin--Drinfeld Structures}
Let $\bE$ be a pinning of $\g$. By an \emph{admissible quadruple} we mean a quadruple $Q=(\Gamma_1,\Gamma_2,\tau,r_\h)$, where $(\Gamma_1,\Gamma_2,\tau)$ is an admissible
triple in the sense of Stolin, and $r_\h\in\h\otimes\h$ satisfies $r_\h+\kappa(r_\h)=\Omega_\h$ and
\[\forall\alpha\in\Gamma_1: (\tau(\alpha)\otimes 1+1\otimes\alpha)(r_\h)=0.\]
Recall that $(\Gamma_1,\Gamma_2,\tau)$ being admissible means that $\Gamma_1,\Gamma_2\subset \Gamma$ and $\tau:\Gamma_1\to\Gamma_2$ is an isometry such that for every $\alpha\in\Gamma_1$ there
exists a positive integer $k$ such that $\tau^k(\alpha)\notin\Gamma_1$.
Associated to these data is the \emph{Belavin--Drinfeld} $r$-matrix
 \[r_\bdr=r_\bdr(\bE,Q)=r_\h+\sum_{\alpha\in\Delta^+} X_\alpha\otimes X_{-\alpha}+\sum_{\substack{\alpha\in\mathrm{Span}_\bbZ(\Gamma_1)^+\\ k>0}}X_\alpha\otimes X_{-\kappa^k(\alpha)}.\]

\begin{Rk} The requirement on $r_\h$ above implies that $r_\bdr$ is defined over $R$.
\end{Rk}

We denote by $\bA_\bdr$ the automorphism group of $\partial r_\bdr$ and by $\bC_\bdr\subseteq \bH$ the centralizer of $r_\bdr$; since the action of $\bH$ on 
$\g\otimes\g$ is linear we have
\[\bC_\bdr(S)=\{h\in\bH(S)|\Ad_h\otimes\Ad_h(r_\bdr)=r_\bdr\}\]
for any $R$-ring $S$, which, since the induced action of $\bH$ on $\h\otimes\h$ is trivial, equals
\[\{h\in\bH(S)|\Ad_h\otimes\Ad_h(r_\bdr')=r_\bdr'\},\]
where $r_\bdr'=r_\bdr-r_\h$ is defined over $\bbQ$. Thus $\bC_\bdr$ is obtained from an affine $\bbQ$-group scheme by base change. Since any $\bbQ$-group scheme is smooth as $\bbQ$ is a field of characteristic zero, 
and since smoothness is preserved by base change, this proves the following.

\begin{Lma} The group $\bC_\bdr$ is smooth.
\end{Lma}

Given an admissible quadruple $Q =(\Gamma_1,\Gamma_2,\tau,r_\h)$, we denote by
$\bA_\Gamma^Q$ the closed subgroup of $\bA_\Gamma$ defined by the equations $\pi(\Gamma_1)=\Gamma_1$, $\pi\tau=\tau\pi$, and $\widehat\pi(r_\h)=r_\h$. 

\begin{Thm}\label{Taut} Let $\bG$ be a split simple adjoint $R$-group with $\g=\Lie(\bG)$. Let $\bE=(\bH,\Gamma,(X_{\alpha})_{\alpha\in\Gamma})$ be a pinning of $\bG$. Fix an admissible
quadruple $Q=(\Gamma_1,\Gamma_2,\tau,r_\h)$ and consider $r_\bdr=r_\bdr(\bE,Q)$. 
Then the sequence \eqref{exact} induces a split exact sequence
\[\xymatrix{
\mathbf{1}\ar[r]&\bC_\bdr\ar[r]^-{\Ad}&\bA_\bdr\ar[r]^-{f}&\bA_{\Gamma}^Q\ar[r]&\mathbf{1}}\]
of affine group schemes.
\end{Thm}

\begin{proof} By \cite[XXXIII.5.5]{SGA3}, there is a (unique) splitting $s:\bA_\Gamma\to\bAut(\bG), \pi\mapsto\widehat\pi$ leaving invariant each of $\bH$ and $\bE$.
Via this splitting, $\bA_\Gamma$ acts on $\bH$ in such a way that for any $R$-ring $S$ and any $h\in\bH(S)$ and $\pi\in\bA_\Gamma(S)$, 
\[\Ad_{\pi\cdot h}=\widehat\pi\Ad_h\widehat\pi^{-1}.\]
If $\pi\in\bA_\Gamma^Q(S)$, then by definition of $\bA_\Gamma^Q$, $\widehat\pi\in\bA_\bdr$, whence for any $h\in\bC_\bdr(S)$ we have $\pi\cdot h\in\bC_\bdr(S)$. 
Hence the action induces an action of $\bA_\Gamma^Q$ on $\bC_\bdr$ and we may form 
the semi-direct product $\bC_\bdr\cdot\bA_\Gamma^Q$ with respect to this action, and we have a group morphism
\[\Ad\times s:\bC_\bdr\cdot\bA_\Gamma^Q\to \bA_\bdr.\]
To prove the theorem it suffices to show that $\Ad\times s$ is an isomorphism. The group $\bC_\bdr$ is smooth (hence flat), and $\bA_\Gamma^Q$ is a closed subgroup of a finite constant
group, hence flat as well. Thus their semi-direct product is flat, and the fiber-wise isomorphism criterion \cite[$_4$17.9.5]{EGAIV} 
reduces the problem to the case where $R=K$ is a 
field of characteristic zero. In that case all group schemes (being of finite type) are smooth, and it suffices, by \cite[22.5]{KMRT}, to show that $(\Ad\times s)(\Kb):\bC_\bdr(\Kb)\cdot\bA_\Gamma^Q(\Kb)\to\bA_\bdr(\Kb)$ 
is an isomorphism of abstract groups. This latter statement holds by the following lemma, which completes the proof.
\end{proof}

\begin{Lma}\label{LKKPS} Let $K$ be an algebraically closed field of characteristic zero. The map $\Ad\times s$ defines an isomorphism of groups $\bC_\bdr(K)\cdot\bA_\Gamma^Q(K)\to\bA_\bdr(K)$. 
\end{Lma}

\begin{proof} The map in question being the restriction of an injective homomorphism, it only remains to be shown that it is surjective. Let $\phi\in\bA_\bdr(K)$. Then $\phi=\widehat\pi\Ad_g$
for some $g\in\bG(K)$ and $\pi\in\bA_\Gamma(K)$. As a first step we will show that necessarily, $g\in\bH(K)$, building on an argument from \cite{KKPS}. Indeed, consider the isomorphism
\[\Xi:\g\otimes\g\to\End_K(g)\]
defined by sending $a\otimes b$ to the linear map $u\mapsto \langle a,u\rangle b$, where $\langle,\rangle$ is the Killing form on $\g$. Now $\Xi((\phi\otimes\phi)(r_\bdr))=\phi\Xi(r_\bdr)\phi^{-1}$, whence
$\phi$ is an automorphism of $\partial r_\bdr$ if and only if $\phi$ commutes with $\Xi(r_\bdr)$, which holds if and only if $\phi$ commutes with its semisimple and nilpotent parts.
Denote by $D$ the semisimple part of $\Xi(r_\bdr)$. It is shown in \cite[Proof of Theorem 1]{KKPS} that the Borel subalgebra $\bo^+$ (resp.\ $\bo^-$) is the normalizer of the eigenspace of $D$ corresponding to the eigenvalue 0 (resp.\ 1).
Thus if $\phi\in\bA_\bdr$, then $\phi$ must preserve $\bo^+$ and $\bo^-$. Since this is true for $\widehat\pi$, it follows that $\Ad_g$ must preserve $\bo^+$ and $\bo^-$. The end of the proof of Theorem 1 of \cite{KKPS} now
applies to yield that this implies that $g\in\bH(K)$.

Next we will show that $\pi\in \bA_\Gamma^Q$. Once this is done, it follows from $(\phi\otimes\phi)(r_\bdr)=r_\bdr$ that $\Ad_g\otimes\Ad_g(r_\bdr)=r_\bdr$, i.e.\ that $g\in\bC_\bdr(K)$, and the proof becomes complete.
Now, $(\Ad_g\otimes\Ad_g)(r_\h)=r_\h$, while $\Ad_g(X_\alpha)=\alpha(g)X_\alpha$ and $\Ad_g(X_{-\alpha})=\alpha(g)^{-1}X_\alpha$. Thus $(\phi\otimes\phi)$ maps $r_\bdr$ to
 \[(\widehat\pi\otimes\widehat\pi)(r_\h)+\sum_{\alpha\in\Delta^+} X_\alpha\otimes X_{-\alpha}+\sum_{\substack{\alpha\in\mathrm{Span}_\bbZ(\Gamma_1)^+\\ k>0}}\lambda_{\alpha,k}(g) X_{\pi(\alpha)}\wedge X_{-\pi\tau^k(\alpha)},\]
where $\lambda_{\alpha,k}(g)=\alpha(g)(\tau^k(\alpha)(g))^{-1}\in K^*$. Thus $(\widehat\pi\otimes\widehat\pi)(r_\h)=r_\h$, and
 \[\sum_{\substack{\alpha\in\mathrm{Span}_\bbZ(\Gamma_1)^+\\ k>0}}{\lambda_{\alpha,k}(g) X_{\pi(\alpha)}\wedge X_{-\pi\tau^k(\alpha)}=\sum_{\substack{\beta\in\mathrm{Span}_\bbZ(\Gamma_1)^+\\ l>0}}X_{\beta}\wedge X_{-\tau^l(\beta)}}.
 \]
Now $\pi$ and $\tau$ map positive roots to positive roots, and the set $\{X_\alpha\}_{\alpha\in\Delta}$ is linearly independent in $\g$. Thus for the above equality to hold, $\pi$ must preserve $\mathrm{Span}_\bbZ(\Gamma_1)^+$ in $\Delta$ and thus it preserves $\mathrm{Span}_\bbZ(\Gamma_1)^+\cap\Gamma=\Gamma_1$. Thus for any $\alpha\in\mathrm{Span}_\bbZ(\Gamma)$, the sum of all terms with first component $X_{\pi(\alpha)}$ in the left-hand sum equals the sum of all terms with first component $X_\beta$ 
in the right-hand sum, for $\beta=\pi(\alpha)$, whence
 \[\sum_{k>0}{\lambda_{\alpha,k}(g) X_{-\pi\tau^k(\alpha)}=\sum_{l>0}}X_{-\tau^l\pi(\alpha)}\]
 for all $\alpha\in\mathrm{Span}_\bbZ\Gamma_1^+$. We want to prove that this implies that $\pi\tau=\tau\pi$. For this it is enough to show that $\pi\tau(\alpha)=\tau\pi(\alpha)$
for all $\alpha\in\Gamma_1$. If $\Gamma_1=\emptyset$, there is nothing to prove. Assume thus that $\Gamma_1\neq\emptyset$. For any $\alpha\in\Gamma_1$, the definition of admissibility implies that there exists a unique integer $l_\alpha\geq 1$ such that $\tau^{l_\alpha}(\alpha)\notin\Gamma_1$
and $\tau^{l_\alpha-1}(\alpha)\in\Gamma_1$. Since $\pi(\Gamma_1)=\Gamma_1$, the above equality and the construction of a Chevalley basis 
implies that all $\lambda_{\alpha,k}=1$, that both sums have $l_\alpha$ terms, and that
\[\{\pi\tau(\alpha),\pi\tau^2(\alpha),\ldots,\pi\tau^{l_\alpha}(\alpha)\}=\{\tau\pi(\alpha),\tau^2\pi(\alpha),\ldots,\tau^{l_\alpha}(\alpha)\}.\]
Since $l_{\tau(\alpha)}=l_\alpha-1$, the set of all $l_\alpha$ is a segment of integers starting at 1. We may thus proceed by induction on $l_\alpha$. If $l_\alpha=1$, the above sets are singletons,
giving the base of the induction. For $l_\alpha>1$, by the induction hypothesis the left hand side is equal to
\[\{\pi\tau(\alpha),\tau\pi\tau(\alpha),\ldots,\tau\pi\tau^{l_\alpha-1}(\alpha)\}.\]
Thus for the above equality of sets to hold, either $\tau\pi(\alpha)=\pi\tau(\alpha)$, or $\tau\pi(\alpha)=\tau\pi\tau^j(\alpha)$ for some $j>0$. This second case is however
not possible, since by injectivity of $\tau\pi$ it implies $\alpha=\tau^j(\alpha)$ for some $j>0$, which violates admissibility. Thus $\pi\tau(\alpha)=\tau\pi(\alpha)$, and the result
follows by induction. This completes the proof.
\end{proof}

\subsection{Standard Structure}
Associated to a pinning $\bE$ is the \emph{Drinfeld--Jimbo} $r$-matrix
\[r_{\drj}=r_{\drj}(\bE)=\frac{1}{2}\Omega_\h+\sum_{\alpha\in\Delta^+} X_\alpha\otimes X_{-\alpha}.\]
In other words $r_\drj=r_\bdr(\bE,Q)$ for the trivial quadruple $Q=(\emptyset,\emptyset,\Id_\emptyset,\frac{1}{2}\Omega_\h)$.

\begin{Lma} If $\bE$ and $\bE'$ are pinnings of $\bG$, then the Lie bialgebras $(\g,\dr_{\drj}(\bE))$ and
$(\g,\dr_{\drj}(\bE'))$ are isomorphic.
\end{Lma}

\begin{proof}
If $\bE$ and $\bE'$ are two pinnings of $\bG$, then there is an automorphism of $\bG$ mapping $\bE$ to $\bE'$ by \cite[XXIII.5.1]{SGA3}. As
$\Omega_\h$ is determined by $\h$ and $\Omega$, there is thus an automorphism $\phi:\g\to\g$ with $(\phi\otimes\phi)(r_{\drj}(\bE))=r_{\drj}(\bE')$, and hence $\phi$ is an isomorphism of Lie bialgebras.
\end{proof}

Following \cite{ES}, we use the term \emph{standard} or \emph{Drinfeld--Jimbo} for the Lie bialgebra structure $\dr_{\drj}(\bE)$ on $\g$.
Up to isomorphism, the above lemma shows that it is independent of the choice of $\bE$, and is thus determined by $\g$. Classifying twisted forms of
standard Lie bialgebras, i.e.\ Lie bialgebras $(\g',\delta')$ such that $(\g'_S, \delta'_S)\simeq(\g_S,(\dr_{\drj})_S)$ for a
split Lie algebra $\g$ with standard Lie bialgebra structure $\dr_{\drj}$ and a faithfully flat $R$-ring $S$ amounts, by Corollary \ref{Cbij}, to describing the set $H^1(S/R,\bA)$, where $\bA$ is the automorphism $R$-group scheme of
$(\g,\dr_{\drj})$. For this, Theorem \ref{Taut} has the following consequence.

\begin{Cor}\label{Caut} Let $\bG$ be a split simple adjoint $R$-group with $\g=\Lie(\bG)$, and let $\bE=(\bH,\Gamma,(X_{\alpha})_{\alpha\in\Gamma})$ be a pinning of $\bG$. 
Then the sequence \eqref{exact} induces a split exact sequence
\begin{equation}\label{exactdj}
\xymatrix{\mathbf{1}\ar[r]&\bH\ar[r]^-{\Ad}&\bA\ar[r]^-{f}&\bA_{\Gamma}\ar[r]&\mathbf{1}} 
\end{equation}
of affine group schemes.
\end{Cor}

\begin{proof} By virtue of Theorem \ref{Taut}, this holds since any $h\in \bH(S)$, with $S$ an $R$-ring, maps $X_\alpha\otimes X_{-\alpha}$ to itself.
\end{proof}

\section{Standard Lie Bialgebras over Fields}
We now specialize to the case where $R=K$ is a field (of characteristic zero), and consider $\Kb/K$-twisted forms of standard Lie bialgebras. Recall that over $K$, 
two (bi)algebras whose underlying modules are finitely generated are locally isomorphic
with respect to the fppf topology if and only if they become isomorphic after scalar extension to $\Kb$. If $\bA$ is an algebraic group over $K$, then $H^1_{\Gal}(K,\bA)$ 
stands in bijection to $K$-isomorphism classes of $\bA$-torsors that become trivial over $\Kb$. 
In order to classify twisted forms of standard Lie bialgebras, we thus wish to compute $H^1_\Gal(K,\bA)$, where $\bA$ is the automorphism group scheme of the standard Lie bialgebra $(\g,\dr_{\drj})$. We can in fact prove the following.

\begin{Thm}\label{Tmain} The map $f$ of \eqref{exactdj} induces an isomorphism of pointed sets
\[f^*: H^1_\Gal(K,\bA)\overset{\sim}{\longrightarrow} H^1_\Gal(K,\bA_{\Gamma}).\]
\end{Thm}

\begin{proof} The split exact sequence of of Theorem \ref{Taut} induces an exact sequence
\[\xymatrix{
H^1_\Gal(K,\bH)\ar[r]^-{\Ad^*}&H^1_\Gal(K,\bA)\ar[r]^-{f^*}&H^1_\Gal(K,\bA_{\Gamma})}\]
of pointed sets. To prove injectivity of $f^*$, let $c\in H^1_\Gal(K,\bA)$. The set of all $c'\in H^1_\Gal(K,\bA)$ with $f^*(c)=f^*(c')$ is in bijection with a quotient of the set $H^1_\Gal(K,_{f^*(c)}\bH)$ by a certain equivalence relation, where $_{f^*(c)}$ indicates a twisted $\Gal (K)$-action.  As $\bA_\Gamma$ acts on $\bH$ via permutation of the roots, $_{f^*(c)}\bH$ is a permutation torus (i.e.\ the Weil restriction of a split torus). Shapiro's Lemma and Hilbert's Theorem 90 then combine to imply that $H^1_\Gal(K,_{f^*(c)}\bH)$ is trivial for each $c$, whence the claim follows.
\end{proof}

An immediate consequence is the following.

\begin{Cor}\label{Ctrivial} Assume that $\g$ is of type $\mathrm{A}_1$, $\mathrm{B}_n$ or $\mathrm{C}_n$ for any $n$, $\mathrm{E}_7$, $\mathrm{E}_8$, $\mathrm{F}_4$ or
$\mathrm{G}_2$. Then all twisted forms of $(\g,\dr_{DJ})$ are isomorphic.
\end{Cor}

\begin{proof} The assumption implies that $\bA_{\Gamma}$ is the trivial group, whence the triviality of $H^1_\Gal(K,\bA_{\Gamma})$ and, by the above theorem,
of $H^1_\Gal(K,\bA)$.
\end{proof}

\section{Scalar Multiples} 
The Belavin--Drinfeld classification is up to equivalence, which, as explained in the introduction, groups together scalar multiples of Lie bialgebra structures on a given Lie algebra.
Therefore, it is reasonable to consider the set $B_\alpha(K)$ of $K$-Lie bialgebra structures on $\g$ that, after scalar extension to $\Kb$, become isomorphic to $\alpha \dr_{\bdr}$ for some $\alpha\in\Kb^*$ and some Belavin--Drinfeld $r$-matrix $r_\bdr$.

We start in a more general setting. Let $\g$ be a split simple Lie algebra over an integral domain $R$ (as always assumed to be a $\bbQ$-ring), let $\delta$ be an $R$-Lie 
bialgebra structure on $\g$, and let $\alpha\in S^*$ with $S$ a faithfully flat $R$-ring. Consider the Lie bialgebra structure $\alpha\delta_S$ on $\g_S$. It is worth noting 
that a priori, there are three possible scenarios for the descent properties of $\alpha\delta_S$. Recall that we wish to consider 
Lie bialgebra structures $\delta'$ on $S/R$-twisted forms $\g'$ of $\g$ such that $(\g'_S,\delta'_S)\simeq(\g_S,\alpha\delta)$. 
We phrase this by saying that $\alpha\delta$ \emph{descends} to $\g'$. Then either $\alpha\delta_S$ descends to $\g$, or it
does not descend to $\g$, but descends to a (non-split) twisted form of $\g$, or $\alpha\delta_S$ does not descend to any form of $\g$.

\begin{Rk} Let $\delta_e$ be an $S$-Lie bialgebra structure on $\g_S$. Then any $R$-Lie bialgebra 
$(\g',\delta')$ with $(\g_S',\delta_S')\simeq (\g_S,\delta_e)$ is $R$-isomorphic to the restriction $\delta_e^\theta$ of $\delta_e$ to
\[\g_S^\theta=\{x\in\g\otimes S| \theta(x\otimes 1)=x\otimes 1\}\]
for some descent datum $\theta$ of Lie algebras.  Indeed if $\theta'$ is the standard descent datum on 
$(\g_S',\delta_S')$ and $\phi:(\g_S',\delta_S')\to (\g_S,\delta_e)$ is an $S$-isomorphism, then $\theta: (\phi\otimes\Id_S)\theta'(\phi^{-1}\otimes\Id_S)$ is a descent datum on $(\g_S,\delta_e)$. 
Thus $\delta_e$ restricts to an $R$-bialgebra structure on  $\g_S^\theta$. This is $R$-isomorphic to $(\g',\delta')$ since by construction of $\theta$, the map $\phi$ is an 
isomorphism of $S$-Lie bialgebras with descent data between $(\g_S',\delta_S')$ with datum $\theta'$ and $(\g_S,\delta_e)$ with datum $\theta$. Thus to classify 
$R$-Lie bialgebras that become isomorphic to $(\g_S,\delta_e)$ over $S$, it suffices to consider restrictions
of $\delta_e$ itself to twisted forms of $\g$. This will be used throughout for the case $\delta_e=\alpha\delta$
\end{Rk}

\subsection{Twisted Cohomology}
The following proposition sheds some light on the occurrence of the possible cases discussed in the opening of this section. Our approach extends that used for so called twisted Belavin--Drinfeld cohomologies, see e.g.\ \cite[Section 7]{KKPS}.
We focus on the case when $R=K$ is a field, and comment on the more general case below.

\begin{Prp}\label{Pfields} Let $\g$ be a split simple Lie algebra over $K$ and let $\delta=\partial r$ for an $r$-matrix $r\in\g\otimes\g$ with 
$r+\kappa(r)=\lambda\Omega$ for some $\lambda\in K^*$. Let $\alpha\in \Kb^*$. Finally let $\g'$ be a twisted form of $\g$ and $(u_\gamma)$ the corresponding $\Gamma$-cocycle, where $\Gamma=\Gal(K)$. Then $\alpha\overline\delta$ descends to $\g'$ if and only if one of the following mutually exclusive conditions holds.
\begin{enumerate}
 \item $\alpha\in K^*$ and $(u_\gamma\otimes u_\gamma)(r)=r$ for each $\gamma\in\Gamma$.
 \item $\alpha^2\in K^*\setminus K^{*2}$ and
 \begin{equation}\label{tau}
(u_\gamma\otimes u_\gamma)(r)=\left\{\begin{array}{ll}
                                         r & \text{if\ } \gamma\in\Gal(K(\alpha))\\
                                         \kappa(r) & \text{if\ } \gamma(\alpha)=-\alpha.
                                        \end{array}\right.
\end{equation}
\end{enumerate}
\end{Prp}

Here we are using the notation $\overline\g=\g_{\Kb}$ and $\overline\delta=\delta_{\Kb}$. 

\begin{proof} Assume that $\alpha\overline\delta$ descends to $\g'$. Then $\overline\delta(x)\in\g'\otimes\g'$ for any $x\in\g'$, which in terms of cocycles implies
\[(u_\gamma\otimes u_\gamma)^{\gamma\otimes\gamma}(\alpha\overline\delta(x))=\alpha\overline\delta(x)\]
for each $\gamma\in\Gamma$. Using the fact that $\delta=\partial r$, that $^{\gamma\otimes\gamma}r=r$ since $r\in\g\otimes\g$, and that $u_\gamma^\gamma x=x$ since $x\in\g'$, this is equivalent to
\[[1\otimes x+x\otimes 1, \gamma(\alpha)(u_\gamma\otimes u_\gamma)(r)]=[1\otimes x+x\otimes 1, \alpha r]\]
for each $\gamma\in\Gamma$. This is in turn equivalent to
\begin{equation}\label{ugamma}
\gamma(\alpha)(u_\gamma\otimes u_\gamma)(r)=\alpha r-\mu_\gamma\Omega 
\end{equation}
for some $\mu_\gamma\in K$. Now $\CYB(\alpha r)=\alpha^2 \CYB(r)=0$ and $\CYB(\gamma(\alpha) (u_\gamma\otimes u_\gamma)(r))=\gamma(\alpha)^2u_\gamma^{\otimes 3}\CYB(r)=0$.
Since
\[(u_\gamma\otimes u_\gamma)(r)+\kappa(u_\gamma\otimes u_\gamma)(r)=(u_\gamma\otimes u_\gamma)(r+\kappa(r))=\Omega,\]
we may, after extending scalars to $\Kb$, apply Lemma \ref{Lelim} with $r_1=\alpha r$, $r_2=\gamma(\alpha)(u_\gamma\otimes u_\gamma)(r)$, $\lambda_1=\alpha\lambda$, $\lambda_2=\gamma(\alpha)\lambda$ and $\mu=\mu_\gamma$. Thus we get $\mu_\gamma=j_\gamma\alpha\lambda$ for some $j_\gamma=0,1$. Since $r+\kappa(r)=\lambda\Omega$, \eqref{ugamma} is then equivalent to
\[\gamma(\alpha)(u_\gamma\otimes u_\gamma)(r)=\alpha (-\kappa)^{j_\gamma}(r).\]
Applying $\kappa$ to both sides gives
\[\gamma(\alpha)(u_\gamma\otimes u_\gamma)\kappa(r)=\alpha(-1)^{j_\gamma}\kappa^{{j_\gamma}+1}(r),\]
and adding the two equations, using $r+\kappa(r)=\lambda\Omega$ and $(u_\gamma\otimes u_\gamma)(\Omega)=\Omega$, yields
\[\gamma(\alpha)\Omega=(-1)^{j_\gamma}\alpha\Omega.\]
Thus $\gamma(\alpha^2)=\alpha^2$ for each $\gamma\in\Gamma$, whence $\alpha^2\in K^*$. For a given $\gamma\in\Gamma$, two cases are possible: either ${j_\gamma}=0$, i.e.\ equivalently $\gamma(\alpha)=\alpha$, and then \eqref{ugamma} gives $(u_\gamma\otimes u_\gamma)(r)=r$; or ${j_\gamma}=1$, i.e.\ equivalently $\gamma(\alpha)=-\alpha$, and then \eqref{ugamma} gives $(u_\gamma\otimes u_\gamma)(r)=\kappa(r)$. It follows that if $\alpha\overline\delta$ descends to $\g'$, then (1) or (2) holds.

Assume, conversely, that (1) or (2) holds. It is straight-forward to check that \eqref{ugamma} is satisfied for each $\gamma\in\Gamma$, with $\mu_\gamma=0$ if $\gamma(\alpha)=\alpha$, and $\mu_\gamma=\alpha\lambda$ if $\gamma(\alpha)=-\alpha$. By the chain of equivalences in the above argument, this implies that
\[(u_\gamma\otimes u_\gamma)^{\gamma\otimes\gamma}(\alpha\overline\delta(x))=\alpha\overline\delta(x).\]
We may then conclude with the lemma below.
\end{proof}

\begin{Lma} Let $\g$ be a finite-dimensional vector space over a field $K$, $(u_\gamma)$ a $\Gamma=\Gal(K)$-cocycle in $\bGL(\g)$ with corresponding twisted form $\g'$, and $\delta$ a Lie coalgebra structure on $\overline\g$. Then $\delta$ descends to $\g'$ if and only if 
\[(u_\gamma\otimes u_\gamma)^{\gamma\otimes\gamma}(\delta(x))=\delta(x)\]
for each $x\in\g'$.
\end{Lma}

\begin{Rk} Of course $\g\simeq \g'$ as $K$-vector spaces. What the lemma achieves is to establish an easily checked condition for when the coalgebra structure $\delta$ is compatible with the Galois action. 
\end{Rk}

\begin{proof} What needs to be shown is that the inclusion $\g'\otimes\g'\subseteq (\overline\g\otimes\overline\g)^{\Gamma'}$ is an equality, where the right hand side denotes the fixed points of $\overline\g\otimes\overline\g$ under the component-wise twisted $\Gamma$-action
\[\gamma\cdot_u(y\otimes z)=u_\gamma(^\gamma y)\otimes u_\gamma(^\gamma z).\]
Let $(e_i)$ be a $K$-basis of $\g'$. It is then finite by assumption, and a $\Kb$-basis of $\overline\g$. Thus if $x\in\overline\g\otimes\overline\g$, then
\[x=\sum_{i,j} \alpha_{ij} e_i\otimes e_j\]
for some $\alpha_{ij}\in \Kb$. Then for each $\gamma\in\Gamma$,
\[\gamma\cdot_u x =\sum_{i,j} \gamma(\alpha_{ij}) u_\gamma(^\gamma e_i)\otimes u_\gamma(^\gamma e_j)=\sum_{i,j} \gamma(\alpha_{ij}) e_i\otimes e_j,\]
where the last equality holds since $e_k\in\g'$ implies $u_\gamma(^\gamma e_k)=e_k$ for each $k$. If $x\in(\overline\g\otimes\overline\g)^{\Gamma'}$, then for each $\gamma\in\Gamma$,
$\gamma\cdot_u x=x$, which, together with the above, by linear independence implies that $\gamma(\alpha_{ij})=\alpha_{ij}$, i.e.\ $\alpha_{ij}\in K$, for each $i$ and $j$. Thus $x\in\g'\otimes\g'$, as desired.
\end{proof}

\begin{Rk} Note that this result encompasses those Lie bialgebras that in \cite{KKPS} and \cite{PS} are treated 
by means of twisted Belavin--Drinfeld cohomologies. Indeed, when they exist, these are obtained by constructing the cocycle $(u_\gamma)$ as follows: 
$u_\gamma=\Id$ for any $\gamma\in\Gal(K(\alpha))$, and $u_{\gamma_\alpha}=\Ad_X^{-1}\phantom{}^{\gamma_\alpha}\Ad_X$ for $X\in\bG(\Kb)$ satisfying certain conditions. 
(Here $\bG$ is an adjoint group with $\g=\Lie\bG$, and $\gamma_\alpha$ is the non-trivial element of $\Gal(K(\alpha)/K)$.) Note that, as the authors remark in \cite{PS}, this 
cocycle is trivial as a Lie algebra cocycle, i.e.\ the fixed locus is the split Lie algebra $\g$. However, the descended Lie bialgebra is, by Lemma \ref{Iso}, not 
isomorphic to $(\g,\beta\partial r)$ for any $\beta\in K^*$.  
\end{Rk}

Led by the above, we define twisted cohomologies as follows.  For each $\alpha\in \Kb^*$ with $\alpha^2\in K^*\setminus K^{*2}$ we write 
$\overline Z^1_\alpha=\overline Z^1(K,\bAut(\g,\partial r),\alpha)$ for the set of all $\Gal (K)$-cocycles $(u_\gamma)$ in $\bAut(\g)$ that 
satisfy \eqref{tau}. Thus $\overline Z^1_\alpha\subset \bAut(\g)(\Kb)$, and we define an equivalence relation $\sim$ on $\overline Z^1_\alpha$ by
\begin{equation}\label{equiv}
(v_\gamma)\sim(u_\gamma) \Longleftrightarrow \exists \rho\in\bAut(\g,\partial r)(\Kb): \forall \gamma: v_\gamma=\rho^{-1}u_\gamma\phantom{}^\gamma\rho 
\end{equation}
and write $\overline H^1_\Gal(K,\bAut(\g,\partial r),\alpha)$ for the set of equivalence classes. The purpose of this set is explained by the following result.

\begin{Prp}\label{PGalois} Let $\alpha\in \Kb^*$. The set of all $K$-Lie bialgebras that 
become isomorphic to $(\overline\g,\alpha\partial r)$ over $\Kb$ is in bijection with
\begin{enumerate}
 \item $H^1_\Gal(K,\bAut(\g,\partial r))$, if $\alpha\in K^*$,
 \item $\overline H^1_\Gal(K,\bAut(\g,\partial r),\alpha)$, if $\alpha^2\in K^*\setminus K^{*2}$, and
 \item the empty set, otherwise.
\end{enumerate}
\end{Prp}

\begin{proof} Corollary \ref{Cbij} and the fact that $\bAut(\g,\alpha\partial r))\simeq\bAut(\g,\partial r))$ whenever $\alpha\in K^*$ together imply (1), and (3) follows from Proposition \ref{Pfields}. 
For (2), the proposition provides a surjective map from the set of all $K$-Lie bialgebras that become isomorphic to $(\overline\g,\alpha\partial r)$
to $\overline Z^1$. To show that it induces a bijection with the twisted cohomology, let first $\g'$ and $\g''$ be twisted forms of 
$\g$ to which $\alpha\overline\delta$ descend (where $\delta=\partial r$), and let $(u_\gamma)$ and $(v_\gamma)$ be the cocycles 
corresponding to $\g'$ and $\g''$, respectively. Denote by $\delta'$ and $\delta''$ the respective descended Lie bialgebra structures. 

If $(\g',\delta')\simeq(\g'',\delta'')$,
then there is a $\Kb$-automorphism $\rho$ of $(\overline\g,\alpha\partial r)$ that maps $\overline\g^{\Gamma_u}$ to $\overline\g^{\Gamma_v}$. (We denote by
$\overline\g^{\Gamma_u}$ the set of all $x\in\overline\g$ that are fixed under the twisted action $\gamma\cdot x=u_\gamma\phantom{}^\gamma x$ of $\Gamma=\Gal(K)$, and 
likewise for $\Gamma_v$.) Then for any $\gamma\in\Gamma$ and $x\in\overline\g^{\Gamma_u}$,
\[\rho(u_\gamma\phantom{}^\gamma x)=\rho(x)=v_\gamma\phantom{}^\gamma\rho(x).\]
Since $\overline\g$ is generated as a $\Kb$-vector space by $\overline\g^{\Gamma_u}$, this implies that $u_\gamma=\rho^{-1}v_\gamma(^\gamma\rho)$, i.e.\ $(u_\gamma)\sim(v_\gamma)$,
since the $\Kb$-automorphisms of $(\overline\g,\alpha\partial r)$ coincide with those of $(\overline\g,\partial r)$. If conversely $(u_\gamma)\sim(v_\gamma)$ with $\rho$ satisfying $u_\gamma=\rho^{-1}v_\gamma(^\gamma\rho)$, then it follows that $\rho$ 
maps $\overline\g^{\Gamma_u}$ to $\overline\g^{\Gamma_v}$ and thus induces an isomorphism $(\g',\delta')\to(\g'',\delta'')$. 
\end{proof}

In the more general case where $R$ is an integral domain, Proposition \ref{Pfields} admits the following generalization. The proof of the following proposition
and the subsequent corollary are analogous to those of the corresponding results over fields. (See the preprint \cite{AP} of this paper for the details.)

\begin{Prp} Assume that $R$ is an integral domain. Let $\g$ be a split simple Lie algebra over $R$, and let $\g'=\g_S^\theta$ be the $S/R$-twisted 
form of $\g$  corresponding to the descent datum $\theta$ and cocycle $\phi$. Let $\delta=\partial r$ be a coboundary $R$-Lie bialgebra 
structure on $\g$, with $r\in\g\otimes\g$ an $r$-matrix satisfying $r+\kappa(r)=\lambda\Omega$ for some $\lambda\in R^*$. Finally let $S$ be a faithfully flat $R$ ring and $\alpha\in S^*$. 
Then $\alpha\delta_S$ descends to $\g'$ if and only if one of the following mutually exclusive conditions holds.
\begin{enumerate}
 \item $\alpha\in R^*$ and $(\phi\otimes \phi)(r\otimes1\otimes1)=r\otimes1\otimes1$, 
 \item $\alpha^2\in R^*\setminus R^{*2}$ and $(\phi\otimes\phi)(r\otimes 1\otimes1)=\kappa(r)\otimes1\otimes 1$.
 \end{enumerate}
\end{Prp}

As in the field case, we can encode this in terms of twisted cohomologies. Given a split simple Lie algebra $g$ over an integral domain $R$, an $r$-matrix $r$ on $\g$, 
a faithfully flat $R$-ring $S$, and $\alpha\in S^*$, we set 
\[\overline Z^1:=\overline Z^1(S/R,\bAut(\g))=\{\phi\in Z^1(S/R,\bAut(\g))|(\phi\otimes\phi)(r)=\kappa(r)\},\]
where $Z^1(S/R,\bAut(\g))$ is the set of 1-cocycles on $\bAut(\g)$ (using the conventions of \cite[17.6]{W}). 
Thus $\overline Z^1\subset \bAut(\g)(S\otimes S)$. We then define an equivalence relation $\sim$ on $\overline Z^1$ by
\[\psi\sim\phi \Longleftrightarrow \exists \rho\in\bAut(\g,\partial r)(S): \psi=(\Id_\g\otimes\kappa)(\rho\otimes\Id_S)(\Id_\g\otimes\kappa)\phi(\rho\otimes\Id_S)^{-1}\]
and write $\overline H^1(S/R,\bAut(\g,\partial r))$ for the set of equivalence classes. 

\begin{Cor} The set of all $R$-Lie bialgebras that 
become isomorphic to $(\g_S,\alpha\partial r)$ over $S$ is in bijection with
\begin{enumerate}
 \item $H^1(S/R,\bAut(\g,\partial r))$, if $\alpha\in R^*$,
 \item $\overline H^1(S/R,\bAut(\g,\partial r))$, if $\alpha^2\in R^*\setminus R^{*2}$, and
 \item the empty set, otherwise.
\end{enumerate}
\end{Cor}

\subsection{Interpreting Twisted Cohomologies}
The twisted cohomologies defined above can be interpreted as (ordinary) cohomologies of twisted groups. To begin with, we need to determine when the twisted cohomology sets are non-empty.
This is the content of the following. Throughout, we work over a field $K$.

\begin{Prp}\label{Ppi} Let $r_\bdr$ be a Belavin--Drinfeld $r$-matrix with associated admissible quadruple $(\Gamma_1,\Gamma_2,\tau,r_\h)$. 
Then $\overline H^1_\Gal(K,\bAut(\g,\partial r_\bdr))$ is non-empty if and only if there exists a diagram automorphism $\pi$ of $\Gamma$ satisfying
\begin{equation}\label{pi}
\begin{array}{lllll}
\pi(\Gamma_1)=\Gamma_2,&\pi(\Gamma_2)=\Gamma_1,&\pi\tau\pi^{-1}=\tau^{-1}&\text{and}& (\widehat\pi\otimes\widehat\pi)(r_\h)=r_\h^{21}.
\end{array}
\end{equation}
\end{Prp}

The following lemma will be helpful.

\begin{Lma}\label{Lpi} Let $r_\bdr$ be the Belavin--Drinfeld $r$-matrix with associated admissible quadruple $(\Gamma_1,\Gamma_2,\tau,r_\h)$. If there exists a diagram automorphism $\pi$
satisfying \eqref{pi}, then there exists a diagram automorphism $\pi'$ of order 1 or 2 satisfying \eqref{pi}. More specifically, either $\pi$ itself is of order 1 or 2, or $\g$ is of
type $\mathrm{D}_4$ and $r_\bdr=r_\drj$, in which case $\Id$ satisfies \eqref{pi}. 
\end{Lma}

\begin{proof} This is immediate if $\g$ is not of type $\mathrm{D}_4$. If $\g$ is of type $\mathrm{D}_4$ and $\Gamma_1\neq\emptyset$, then it is straight forward to check, case by case,
that for all admissible $\Gamma_1$, $\Gamma_2$ and $\tau$, a diagram automorphism satisfying \eqref{pi} must satisfy $\pi^2=\Id$. If $\Gamma_1=\Gamma_2=\emptyset$ and $\pi^2\neq\Id$,
then $\pi^3=\Id$. But then the condition $(\widehat\pi\otimes\widehat\pi)(r_\h)=r_\h^{21}$ implies that $(\widehat\pi^2\otimes\widehat\pi^2)(r_\h)=r_\h$, whence
\[r_\h=(\widehat\pi^3\otimes\widehat\pi^3)(r_\h)=(\widehat\pi\otimes\widehat\pi)(r_\h)=r_\h^{21}.\]
Thus since $r_\h+r_\h^{21}=\Omega_\h$, this implies that $r_\h=\frac{1}{2}\Omega_\h$, whence $r_\bdr=r_\drj$. In that case $\Id$ satisfies \eqref{pi}, and the proof is complete.           
 
\end{proof}

\begin{proof}[Proof of Proposition \ref{Ppi}] Assume such an element $\pi$ exists. By the above lemma, we may assume $\pi^2=\Id$. Let $\chi$ be the Chevalley automorphism of $\g$
and set $\phi=\chi\pi$, which we view as an element of $\bAut(\g)(\Kb)$. We claim that 
\[\phi\in\overline Z^1(K,\bAut(\g,\partial r_\bdr),\alpha).\] 
We have $\chi(h)=-h$ for any $h\in\h$, whence $(\chi\otimes\chi)(r_\h)=r_\h$. Moreover $\pi$ permutes all $\alpha\in\Delta^+$, and $\pi\tau^k=\tau^{-k}\pi$, whence
\[(\phi\otimes\phi)(r_\bdr)=r_\h^{21}+\sum_{\alpha\in\Delta^+}X_{-\alpha}\otimes X_\alpha+\sum_{\substack{\alpha\in\mathrm{Span}_\bbZ(\Gamma_1)^+\\ k>0}}X_{-\pi(\alpha)}\wedge X_{\tau^{-k}\pi(\alpha)}.\]
For fixed $\alpha$ and $k$, setting $\alpha'=\tau^{-k}\pi(\alpha)$, the term $X_{-\pi(\alpha)}\wedge X_{\tau^{-k}\pi(\alpha)}$ becomes $X_{-\tau^k(\alpha')}\wedge X_{\alpha'}$. Summing over $\alpha'$ and $k$, we thus get
\[(\phi\otimes\phi)(r_\bdr)=r_\h^{21}+\sum_{\alpha\in\Delta^+}X_{-\alpha}\otimes X_\alpha+\sum_{\substack{\alpha'\in\mathrm{Span}_\bbZ(\Gamma_1)^+\\ k>0}}X_{-\tau^k(\alpha')}\wedge X_{\alpha'},\]
which is equal to $r_\bdr^{21}$. Since $\chi$ and $\pi$ commute and are of order at most two, we have $\phi^2=1$. Since $\phi$ is stable under the $\Gal(K)$-action, we get an 
element $u_\gamma$ of $Z^1(K,\bAut(\g))$ by setting
\[
u_\gamma=\left\{\begin{array}{ll}
                                         \Id & \text{if\ } \gamma\in\Gal(K(\alpha))\\
                                         \phi & \text{if\ } \gamma(\alpha)=-\alpha.
                                        \end{array}\right.\]
By construction, this element belongs to $\overline Z^1(K,\bAut(\g,\partial r_\bdr),\alpha)$.

Conversely, assume that $\overline Z^1(K,\bAut(\g,\partial r_\bdr),\alpha)$ is non-empty. Then in particular there exists $\phi\in\bAut(\g,\partial r_\bdr)(\Kb)$ satisfying 
$(\phi\otimes\phi)(r_\bdr)=r_\bdr^{21}$. Arguing as in \ref{LKKPS}, we conclude that any such $\phi$ must map the semisimple part of $\Xi(r_\bdr)$ to the semisimple part of $\Xi(r_\bdr^{21})$.
By \cite{KKPS}, the semisimple part of $\Xi(r_\bdr)$ differs from $\Xi(r_\drj)$ only by an element in $\Xi(\h\otimes\h)$, and thus likewise for $\Xi(r_\bdr^{21})$. Thus we have
$(\phi\otimes\phi)(r_\drj)=r_\drj^{21}$, and thence $\phi=\chi\Ad_h\pi$ for some $h\in\bH(\Kb)$ and a diagram automorphism $\pi$. Thus
\[(\pi\otimes\pi)(r_\bdr)=\Ad_{h^{-1}}\chi(r_\bdr^{21}).\]
This implies $(\pi\otimes\pi)(r_\h)=r_\h^{21}$ (since both $\Ad_h$ and $\chi$ leave $\h\otimes\h$ fixed), and
\[\sum_{\substack{\alpha\in\mathrm{Span}_\bbZ(\Gamma_1)^+\\ k>0}}X_{\pi(\alpha)}\wedge X_{-\pi\tau^k(\alpha)}=\sum_{\substack{\beta\in\mathrm{Span}_\bbZ(\Gamma_1)^+\\ l>0}}\lambda_{\beta,l}X_{\tau^l(\beta)}\wedge X_{-\beta},\]
for non-zero scalars $\lambda_{\beta,l}$, which therefore have to equal 1. Thus $(\pi\otimes\pi)(r_\bdr)=\chi(r_\bdr^{21})$, and $h^{-1}\in\bC(\chi(r_\bdr^{21}))$, which implies that
$h\in\bC(r_\bdr)$. Similarly to the case in the proof of Lemma \ref{LKKPS} this implies that $\pi(\Gamma_1)=\Gamma_2$ and $\pi(\Gamma_2)=\Gamma_1$. Relabeling the terms, the equality becomes
\[\sum_{\substack{\alpha\in\mathrm{Span}_\bbZ(\Gamma_1)^+\\ k>0}}X_{\pi(\alpha)}\wedge X_{-\pi\tau^k(\alpha)}=\sum_{\substack{\beta\in\mathrm{Span}_\bbZ(\Gamma_1)^+\\ l>0}}X_{\beta}\wedge X_{-\tau^{-l}(\beta)},\]
Proceeding again as in the proof of Lemma \ref{LKKPS}, this implies that for each $\alpha\in\Gamma_1$ we have the equality of sets
\[\{\pi\tau(\alpha),\ldots,\pi\tau^{l_\alpha}(\alpha)\}=\{\tau^{-1}\pi(\alpha),\ldots,\tau^{-l_\alpha}\pi(\alpha)\},\]
where $l_\alpha$ is defined as in the proof of Lemma \ref{LKKPS}. As there, we proceed by induction on $l_\alpha$, the case $l_\alpha=1$ being clear. 
For $l_\alpha>1$, since $l_{\tau(\alpha)}=l_\alpha-1$, by the induction hypothesis the left hand side is equal to
\[\{\pi\tau(\alpha),\tau^{-1}\pi\tau(\alpha),\ldots,\tau^{-1}\pi\tau^{l_\alpha-1}(\alpha)\}.\]
Thus $\tau^{-1}\pi(\alpha)=\pi\tau(\alpha)$, since by admissibility $\tau^{-1}\pi(\alpha)\neq\tau^{-1}\pi\tau^j(\alpha)$ for any $j>0$. This completes the proof.
\end{proof}

We are now ready to re-interpret $\overline H^1(K,\bAut(\g,\partial r_\bdr))$ whenever it is non-empty. 
Recall that in these cases there exists a diagram automorphism $\pi$ of order at most two satisfying \eqref{pi}. Recall also that we are in the situation where $K\subset K(\alpha)\subseteq\Kb$,
with $\alpha^2\in K$. 

Set $\bA_\bdr=\bAut(\g,\partial r_\bdr)$. We have a map $v=v_\pi:\Gal(K)\to \bAut(\bA_\bdr)(\Kb)$ defined as the identity on $\Gal(K(\alpha))$ and by mapping
$\gamma_\alpha$ to $\rho\mapsto \chi\pi\rho(\chi\pi)^{-1}=\chi\pi\rho\pi\chi$: this is well defined since by the above, $\chi\pi(r_\bdr)=r_\bdr^{21}$, so
\[(\chi\pi\rho(\chi\pi)^{-1}\otimes\chi\pi\rho(\chi\pi)^{-1})(r_\bdr)=r_\bdr.\]
Since $\Gal(K)$ acts trivially on $\rho\mapsto \chi\pi\rho(\chi\pi)^{-1}$, and since the map is of order 2, it is a cocycle, and, following \cite{KPS} we may consider the twisted group
$(\bA_\bdr)_v$. The $\Gal(K)$-action defining the cocycle set $Z^1_\Gal(K,(\bA_\bdr)_v)$ and the cohomology $H^1_\Gal(K,(\bA_\bdr)_v)$ is given by
\[\gamma\cdot\rho=v(\gamma)(^\gamma\rho).\]

\begin{Thm}\label{Ttwisted} Let $r_\bdr$ be a Belavin--Drinfeld $r$-matrix with associated admissible quadruple $(\Gamma_1,\Gamma_2,\tau,r_\h)$ such that 
$\overline H^1(K,\bA_\bdr)$ is non-empty, and let $\pi$ be any diagram automorphism of $\Gamma$ of order at most two satisfying \eqref{pi}. Then the map
\[\overline Z^1(K,\bA_\bdr,\alpha)\to Z^1_\Gal(K,(\bA_\bdr)_v)\]
defined by mapping the cocycle $(u_\gamma)$ to the cocycle $\widehat u_\gamma$ defined by
\[\widehat u_\gamma=\left\{\begin{array}{ll}
                                         u_\gamma & \text{if\ } \gamma\in\Gal(K(\alpha))\\
                                         u_\gamma\pi\chi & \text{if\ } \gamma(\alpha)=-\alpha,
                                        \end{array}\right.\]
induces an injective map
\[\overline H^1(K,\bA_\bdr)\to H^1_\Gal(K,(\bA_\bdr)_v),\]
where $v=v_\pi$ is constructed as above.
\end{Thm}

\begin{proof} First we show that the map is well-defined. If $(u_\gamma)\in\overline Z^1(K,\bA_\bdr,\alpha)$ we need to show that $(\widehat u_\gamma)$ is in $\bA_\bdr(\Kb)$ and satisfies the twisted cocycle condition
\[\widehat u_{\gamma_1\gamma_2}=\widehat u_{\gamma_1}v(\gamma_1)(^{\gamma_1} \widehat u_{\gamma_2}).\]
To show that each $\widehat u_\gamma$ is an automorphism we need to show that $(\widehat u_\gamma\otimes\widehat u_\gamma)(r_\bdr)=r_\bdr$. This is automatic whenever $\gamma(\alpha)=\alpha$, and holds when $\gamma(\alpha)=-\alpha$ by the proof of Proposition \ref{Ppi}.
The twisted cocycle condition follows from the cocycle condition on $(u_\gamma)$ by a direct computation in each of
the four cases $(\gamma_1(\alpha),\gamma_2(\alpha))=(\pm\alpha,\pm\alpha)$. Thus the map is well defined. To show that it induces an injective map on cohomology, we need show that
\[(u_\gamma)\sim(w_\gamma) \Longleftrightarrow (\widehat u_\gamma)\sim(\widehat w_\gamma),\]
where the equivalence relations are in the respective cohomology sets. For $\gamma\in\Gal(K(\alpha))$, this holds by definition. For those $\gamma\in\Gal(K)$ with
$\gamma(\alpha)=-\alpha$, the right hand equivalence amounts to
\[w_\gamma\pi\chi=\rho^{-1}u_\gamma\pi\chi\chi\pi\phantom{}^\gamma\rho\pi\chi\]
which is equivalent to
\[w_\gamma=\rho^{-1}u_\gamma\phantom{}^\gamma\rho,\]
which is precisely what the left hand equivalence amounts to. This completes the proof.
\end{proof}

\section{Previous Results Revisited}
In the light of the above, we will now review those results obtained in \cite{PS} (for split Lie algebras) and \cite{AS} (for a class of non-split Lie algebras) which are concerned with Drinfeld--Jimbo Lie bialgebra structures. Throughout, we work over a field $K$ of characteristic zero and consider($\Kb/K$)-twisted forms. We begin by the following consequence of Theorem \ref{Tmain}.

\begin{Prp}\label{Pinj} Let $\g$ be a split simple Lie algebra over $K$, and let $\g'$ be a twisted form of $\g$. Then there is, up to $K$-isomorphism, at most one Lie bialgebra structure $\delta'$ on $\g'$ such that $(\g',\delta')$ is a twisted form of the standard Lie bialgebra structure on $\g$. 
\end{Prp}

\begin{proof} The inclusion $i:\bA\to\bAut(\g)$ induces a map 
\[i^*: H^1_\Gal(K,\bA)\to H^1_\Gal(K,\bAut(\g)).\] 
Now the first (resp.\ second) of these cohomology sets classifies those Lie bialgebras (resp.\ Lie algebras) that are twisted forms of $(\g,\partial r_\drj)$ (resp. of $\g$), and the map $i^*$ corresponds to sending the isomorphism class of a Lie bialgebra to the isomorphism class of the underlying Lie algebra. We thus need show that $i^*$ is injective. But by construction, the isomorphism
$f^*: H^1_\Gal(K,\bA)\overset{\sim}{\longrightarrow} H^1_\Gal(K,\bA_{\Gamma})$ of Theorem \ref{Tmain} factors through $i^*$, which therefore is injective.
\end{proof}

\begin{Rk} Note that in general, the map $i^*$ is not surjective, meaning that there exist twisted forms $\g'$ of $\g$ which admit no Lie bialgebra structure that is a twisted form of the standard structure on $\g$. By corollary \ref{Ctrivial}, this is in particular the case whenever $\g'$ is non-split of type $\mathrm{A}_1$, $\mathrm{B}_n$, $\mathrm{C}_n$, $\mathrm{E}_7$, $\mathrm{E}_8$, $\mathrm{F}_4$ or $\mathrm{G}_2$.
\end{Rk}

\begin{Rk} In \cite{PS}, this was proved in the special case $\g'=\g$, by formulating the problem in terms of Galois cohomology of split tori and using Steinberg's theorem. This essentially corresponds to considering Lie bialgebra structures on $\g$ up to those isomorphisms that are inner automorphisms of $\g$. 
\end{Rk}

\begin{Rk} In \cite{AS}, the authors studied the classification of Lie bialgebra structures on certain non-split Lie algebras of type $\mathrm{A}_n$, up to equivalence by certain natural gauge groups. More precisely, the authors considered special unitary Lie algebras under the action of unitary groups with respect to a non-square $d\in K$. Curiously, for twisted forms of the standard structure, they showed that if such twisted forms exist, there exists a unique equivalence class if $n$ is odd, but  if $n$ is even, the equivalence classes are parametrized by $K/N(K(\sqrt{d}))$. This does not contradict the above result, since the unitary group is not adjoint. It is rather straight forward to check that if one uses the corresponding adjoint group in the calculations of \cite{AS}, one obtains uniqueness in all cases. 
\end{Rk}

The main part of \cite{AS} is concerned with Lie bialgebras that upon extension to $\Kb$ become equivalent to $\alpha\partial r_\drj$ for some $\alpha\in \Kb^*\setminus K^*$ with $d:=\alpha^2\in K^*$. 

To review these results, let $K$ be a field of characteristic zero admitting a quadratic extension $L=K(\sqrt d)\supsetneq K$ with $d\in K^*$. Set $\g=\sll_n
(K)$ and consider the Lie algebra
\[\g'=\su_n(K,d)=\{x\in\sll_n(L)|\overline x^t=-x\},\]
where $x\mapsto \overline x$ is the linear map induced by the non-trivial Galois automorphism of $L/K$.

Note that $\g'_{L}\simeq\g_{L}=\sll_n(L)$. In the sequel we fix an algebraic closure $\Kb$ of $K$ containing $L$, so that $\g'_{\Kb}\simeq\sll
_n(\Kb)$. As a direct consequence of Proposition \ref{Pfields}, we obtain the following result from \cite{AS}.

\begin{Cor} Let $\alpha\in \Kb^*$. If the Lie bialgebra structure $\alpha\partial r_\drj$ on $\sll_n(\Kb)$ descends to $\g'$, then $\alpha^2\in K
^*$. 
\end{Cor}

One thus distinguishes two cases: $\alpha^2\in K^*d$, in which case $K(\alpha)=L$, and $\alpha^2\notin K^*d$, in which case $K(\alpha)\cap L=K$. 
The former case is the one that is thoroughly studied in \cite{AS}, and in this case, we obtain the following result. 

\begin{Cor} Let $\alpha\in \Kb^*$ with $\alpha^2\in K^*d$. Then $\alpha\partial r_\drj$ descends to $\g'$.
\end{Cor}

\begin{proof} Note that the twisted form $\g'$ of $\g$ corresponds to the $\Gal(K)$-cocycle $(u_\gamma)$ defined by $u_{\gamma_\alpha}(x)
=-x^t$, where $\gamma_\alpha$ is the non-identity element of $\Gal(L/K)$, and by $u_\gamma=\Id$ whenever $\gamma\in\Gal(L)$. Since $r_\drj^{t\otimes t}=\kappa(r_\drj)$, Proposition \ref{Pfields} implies that $\alpha\partial r_\drj$ descends
to $\g'$. 
\end{proof}

\begin{Rk} In \cite{AS}, the authors obtain, over e.g.\ fields of cohomological dimension at most 2, a 
classification parametrized by $(K^*/N(L^*))^m$ for a certain power $m$. Although we are here considering \emph{isomorphism classes} of Lie bialgebras that 
become a unique \emph{isomorphism} class after scalar extension, whereas in \cite{AS} the authors consider \emph{equivalence} classes that become a unique 
\emph{equivalence} class upon extension, our results above can be used to explain the appearance of these norm classes, namely by restricting ourselves to inner automorphisms in Proposition \ref{PGalois}. 
Let us first clarify what we mean. The embedding $\g'\to\g_L$ defines an $L$-isomorphism of Lie algebras $\g'_L\to\g_L$. A Lie bialgebra structure $\delta'$ on $\g'$ induces, by means of this isomorphism, an $L$-Lie bialgebra structure $\delta'_L$ on $\g_L$. We shall consider those $\delta'$ where $\delta'_L\simeq\sqrt{d}\partial r_\drj$ via an 
inner automorphism of $\g_L$, i.e.\ an element of the form $\Ad_X$ for $X\in\bGL(L)$. Two such structures on $\g'$ are considered (gauge) equivalent if they are isomorphic via an inner automorphism of $\g'$. Recall further that an inner automorphism of $\g_L$ is an automorphism of $(\g_L,\partial r_\drj)$ if and only if it is of the form $\Ad_D$ with $D\in\bH_n(L)$, where $\bH$ is the split torus of $\bGL_n$ fixed by the choice of a pinning.

This corresponds to considering, in $\overline Z^1_\alpha$, those cocycles $(u_\gamma)$ with $u_{\gamma_\alpha}=\Ad_{D_u}\chi$ for 
the generator $\gamma_\alpha$ of $\Gal(L/K)$ with $D_u\in\bH_n(L)$, where $\chi$ is the Chevalley automorphism. 
(Note that in general, not every $D$ satisfies $\Ad_D\chi\in\overline Z^1_\alpha(L)$, as the map in Theorem \ref{Ttwisted} may 
not be surjective. We will not go into details, but refer to \cite{AS}.) Two cocycles $\Ad_{D_u}\chi$ and $\Ad_{D_v}\chi$ 
are then equivalent if they satisfy \eqref{equiv} with $\rho=\Ad_D$ for some $D\in\bH(L)$. This equivalence condition translates as
\[\Ad_{D_v}\chi=\Ad_D^{-1}\Ad_{D_u}\chi\Ad_{\overline D}.\]
By definition of $\chi$, and the fact that $D^t=D$, this is equivalent to
\[\Ad_{D_v}\chi=\Ad_D^{-1}\Ad_{D_u}\Ad_{\overline D}^{-1}\chi,\]
i.e.\ $\Ad_{D_u}=\Ad_{D\overline D D_v}$, which amounts to saying that each entry of $D_v$ differs from the corresponding entry of $D_u$ only by an element of $N_{L/K}(L^*)$. We thus 
retrieve the result obtained in \cite{AS}. (Technically, $\overline Z^1_\alpha$ was defined with respect to the extension $\Kb/K$, but one can similarly consider any finite Galois 
extension.)
\end{Rk}


\begin{thebibliography}{99}
\bibitem[AP]{AP} S.\ Alsaody and A.\ Pianzola, On the Classification of Lie Bialgebras by Cohomological Means (Preprint version of this article). arXiv:1810.05288.
\bibitem[AS]{AS} S.\ Alsaody and A.\ Stolin, Lie bialgebras, fields of cohomological dimension at most 2 and Hilbert's seventeenth problem. J.\ Algebra 476, 368--394 (2017).
\bibitem[BD]{BD} A.\ Belavin and A., V.\ Drinfeld, Triangle equations and simple Lie algebras. Soviet Sci.\ Rev.\ Sect.\ C: Math.\ Phys.\ Rev.\ 4, 93--165 (1984).
\bibitem[BFS]{BFS} K.\ I.\ Beidar, Y.\ Fong and A.\ Stolin, On Frobenius algebras and the quantum Yang--Baxter equation. Trans.\ Amer.\ Math.\ Soc.\ 349, 3823--3836 (1997).
\bibitem[D]{D} V.\ Drinfeld, On some unsolved problems in quantum group theory. Quantum groups (Leningrad, 1990), 1--8, Lecture Notes in Math.\ 1510, Springer,
Berlin (1992).
\bibitem[EK1]{EK1} P.\ Etingof and D.\ Kazhdan, Quantization of Lie bialgebras I. Sel.\ Math.\ (NS) 2, 1--41 (1996).
\bibitem[EK2]{EK2} P.\ Etingof and D.\ Kazhdan, Quantization of Lie bialgebras II. Sel.\ Math.\ (NS) 4, 213--232 (1998).
\bibitem[EGAIV]{EGAIV} A.\ Grothendieck (avec la collaboration de J.\ Dieudonn\'e), \emph{El\'ements de G\'eom\'etrie Alg\'ebrique IV}, Publications 
math\'ematiques de l'I.H.\'E.S. 20, 24, 28 and 32 (1964 - 1967).
\bibitem[ES]{ES} P.\ Etingof and O.\ Schiffmann, \emph{Lectures on Quantum Groups}. International Press, Somerville, MA (2002).
\bibitem[KKPS]{KKPS} B.\ Kadets, E.\ Karolinsky, A.\ Stolin and I.\ Pop, Classification of quantum groups and Belavin--Drinfeld cohomologies. Commun.\ Math.\ Phys.\ 344, 1--24 (2016). 
\bibitem[KPS]{KPS} E.\ Karolinsky, A.\ Pianzola and A.\ Stolin, Classification of Quantum Groups via Galois Cohomology. Preprint, arXiv:1806.05640 (2018).
\bibitem[KMRT]{KMRT} M.-A.\ Knus, A.\ Merkurjev, M.\ Rost and J.-P.\ Tignol, \emph{The Book of Involutions}. AMS Colloquium Publications \textbf{44}, American Mathematical Society (1998).
\bibitem[PS]{PS} A.\ Pianzola and A.\ Stolin, Belavin--Drinfeld solutions to the Yang--Baxter equation: Galois cohomology considerations. 
Bull.\ Math.\ Sci.\ 8, 1--14 (2018).
\bibitem[SGA3]{SGA3} {\it S\'eminaire de G\'eom\'etrie alg\'ebrique de l'I.\ H.\ E.\ S., 1963-1964, sch\'emas en groupes, dirig\'e par M.\ Demazure et A.\ Grothendieck},  Lecture Notes in Math. 151-153. Springer (1970).
\bibitem[W]{W} W.\ Waterhouse, \emph{Introduction to Affine Group Schemes}. Graduate Texts in Mathematics 66, Springer-Verlag, New York (1979).
\end{thebibliography}
\end{document}